\documentclass[11pt]{amsart}
\usepackage{graphicx}
\usepackage[mathcal]{euscript}
\usepackage{float}
\usepackage{cite}
\usepackage{verbatim}
\usepackage{url}
\usepackage{tikz}
\usepackage{adjustbox}
\usepackage{amsmath}
\usepackage{multicol}
\usepackage{rotating}
\usepackage{enumitem}
\usepackage{subcaption}

\makeatletter
\def\bbordermatrix#1{\begingroup \m@th
  \@tempdima 4.75\p@
  \setbox\z@\vbox{%
    \def\cr{\crcr\noalign{\kern2\p@\global\let\cr\endline}}%
    \ialign{$##$\hfil\kern2\p@\kern\@tempdima&\thinspace\hfil$##$\hfil
      &&\quad\hfil$##$\hfil\crcr
      \omit\strut\hfil\crcr\noalign{\kern-\baselineskip}%
      #1\crcr\omit\strut\cr}}%
  \setbox\tw@\vbox{\unvcopy\z@\global\setbox\@ne\lastbox}%
  \setbox\tw@\hbox{\unhbox\@ne\unskip\global\setbox\@ne\lastbox}%
  \setbox\tw@\hbox{$\kern\wd\@ne\kern-\@tempdima\left[\kern-\wd\@ne
    \global\setbox\@ne\vbox{\box\@ne\kern2\p@}%
    \vcenter{\kern-\ht\@ne\unvbox\z@\kern-\baselineskip}\,\right]$}%
  \null\;\vbox{\kern\ht\@ne\box\tw@}\endgroup}
\makeatother

\usepackage{blkarray, bigstrut, graphicx}

\usetikzlibrary{decorations.pathreplacing}
\usetikzlibrary{arrows.meta}

\usetikzlibrary{arrows, shapes, decorations, decorations.markings, backgrounds, patterns, hobby, knots, calc, positioning}

\restylefloat{figure}

\newtheorem{theorem}{Theorem}[section]

\newtheorem*{corollaryconj*}{Corollary of Conjecture \ref{conjecture:Signature}}

\theoremstyle{definition}

\newtheorem{example}[theorem]{Example}
\newtheorem{question}[theorem]{Question}

\theoremstyle{remark}

\numberwithin{equation}{section}

\def\MRI{\operatorname{MRI}}
\def\Reg{\operatorname{Reg}}

\usepackage{tikz}

\usetikzlibrary{arrows,shapes,decorations,backgrounds,patterns}

\pgfdeclarelayer{background}
\pgfdeclarelayer{background2}
\pgfdeclarelayer{background2a}
\pgfdeclarelayer{background2b}
\pgfdeclarelayer{background3}
\pgfdeclarelayer{background4}
\pgfdeclarelayer{background5}
\pgfdeclarelayer{background6}
\pgfdeclarelayer{background7}

\pgfsetlayers{background7,background6,background5,background4,background3,background2b,background2a,background2,background,main}

\begin{document}

\tikzset{invclip/.style={clip,insert path={{[reset cm]
      (-16383.99999pt,-16383.99999pt) rectangle (16383.99999pt,16383.99999pt)
    }}}}

\title{The multi-region index of a knot}

\author[S. Goodhill]{Sarah Goodhill}
\address{Department of Mathematics and Statistics\\
Vassar College\\
Poughkeepsie, NY}
\email{sgoodhill@vassar.edu}

\author[A. M. Lowrance]{Adam M. Lowrance}
\address{Department of Mathematics and Statistics\\
Vassar College\\
Poughkeepsie, NY} 
\email{adlowrance@vassar.edu}

\author[V. Munoz Gonzales]{Valeria Munoz Gonzales}
\address{Department of Mathematics and Statistics\\
Vassar College\\
Poughkeepsie, NY}
\email{vmunozgonzales@vassar.edu}

\author[J. Rattray]{Jessica Rattray}
\address{Department of Mathematics and Statistics\\
Vassar College\\
Poughkeepsie, NY}
\email{jrattray@vassar.edu}

\author[A. Zeh]{Amelia Zeh}
\address{Department of Mathematics and Statistics\\
Vassar College\\
Poughkeepsie, NY}
\email{azeh@vassar.edu}

\thanks{This paper is the result of a summer research project in Vassar College’s Undergraduate Research Science Institute
(URSI). The second author is supported by NSF grant DMS-1811344.}

\date{}

\begin{abstract}
Using region crossing changes, we define a new invariant called the multi-region index of a knot. We prove that the multi-region index of a knot is bounded from above by twice the crossing number of the knot. In addition, we show that the minimum number of generators of the first homology of the double branched cover of $S^3$ over the knot is strictly less than the multi-region index. Our proof of this lower bound uses Goeritz matrices.
\end{abstract}

\maketitle

\section{Introduction}
A \textit{region} of a knot diagram $D$ is a connected component of $S^2-|D|$ where $|D|$ is the $4$-regular graph obtained by forgetting the crossing information of $D$. A \textit{region crossing change} at a region $R$ of a knot diagram $D$ is the transformation that changes all of the crossings in the boundary of $R$; see Figure \ref{figure:RegionChange}. Shimizu \cite{Shimizu} proved that any knot diagram $D$ has a set of regions such that performing a region crossing change on all the regions in the set results in a diagram of the unknot, and thus region crossing change is an unknotting operation. However, region crossing change is not an unlinking operation, as one can see when considering the standard diagram of the Hopf link. In this standard diagram, each of the four regions have both crossings in their boundary, and thus changing any region flips the diagram to its mirror image. Cheng and Gao \cite{CG} and Cheng \cite{Cheng} classified the link diagrams for which region crossing change is an unlinking diagram. Dasbach and Russell \cite{DR} examined region crossing changes for link diagrams on oriented surfaces. There are even games played on knot diagrams where the moves each player makes are region crossing changes \cite{Game2,Game}.

There are several ways to measure the difficulty of unknotting a knot using region crossing changes. Aida \cite{Aida} showed that every knot has a diagram such that a single region crossing change transforms the diagram into the unknot, and so just counting the minimum over all diagrams of the number of region crossing changes needed to transform a knot into the unknot is not an interesting invariant. Shimizu \cite{Shimizu} defined the \textit{region unknotting number} of a knot $K$ to be the minimum number of region crossing changes in any minimal crossing diagram of $K$ needed to transform the diagram into a diagram of the unknot. In another direction, Kawauchi, Kishimoto, and Shimizu \cite{KKS} defined the \textit{region index} of a knot to be the fewest number of crossings in any region $R$ of a diagram of $K$ where performing a region crossing change on $R$ results in the unknot. 

In this paper, we define another invariant that measures the difficulty of unknotting via region crossing changes. Suppose that the set of regions of $D$ is $\mathcal{R}=\{R_0,\dots,R_n\}$, and that region $R_i$ has $c(R_i)$ crossings in its boundary for $i=0,\dots, n$. A \textit{set of unknotting regions} in a knot diagram $D$ is a set of regions of $D$ such that performing region crossing changes on all of the regions results in a diagram of the unknot. The \textit{multi-region index} $\MRI(D)$ of the knot diagram $D$ is defined as \begin{equation}
    \label{equation:MRI(D)}
\MRI(D)=\min\left\{\sum_{j={1}}^{k} c(R_{i_j})~:~\{R_{i_1},\dots,R_{i_k}\}\text{ is a set of unknotting regions of }D \right\},
\end{equation}
that is the multi-region index of $D$ is the minimum total number of crossings changed in any set of unknotting regions of $D$ where a crossing is counted multiple times if it appears in the boundary of multiple regions in the set of unknotting regions. The \textit{multi-region index} $\MRI(K)$ of a knot $K$ is the minimum of $\MRI(D)$ over all diagrams $D$ of $K$.

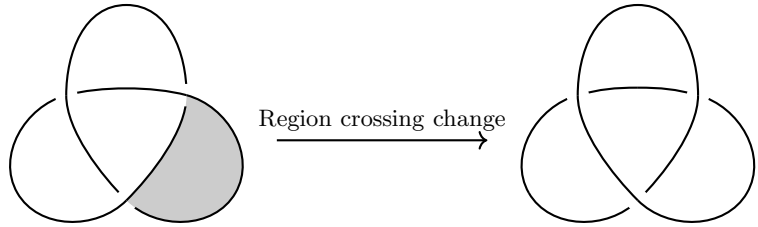
\begin{figure}[h!]
\[\begin{tikzpicture}[outer sep = 0mm, scale = .4]

\def\gap{.15cm}

\coordinate (A) at (-3.625,.75);
\coordinate (B) at (-2,4);
\coordinate (C) at (0,7);
\coordinate (D) at (0,.5);
\coordinate (E) at (2,4);
\coordinate (F) at (3.625,.75);

\fill[white!80!black] (E) to [out = -15, in = 60] (F) to [out = 240, in=-45] (D) to [out = 45, in = 270, looseness = .75] (E);

\begin{scope}
	\begin{pgfinterruptboundingbox} 
		\path [invclip] (E) circle (\gap);
	\end{pgfinterruptboundingbox}
	\draw[thick] (E) to [out = 90, in = 0] (C) to [out = 180, in = 90] (B);
\end{scope}

\begin{scope}
\begin{pgfinterruptboundingbox} 
		\path [invclip] (D) circle (\gap);
	\end{pgfinterruptboundingbox}
	\draw[thick] (B) to [out = 270, in = 135, looseness = .75] (D);
\end{scope}

\begin{scope}
	\begin{pgfinterruptboundingbox} 
		\path [invclip] (E) circle (\gap);
	\end{pgfinterruptboundingbox}
	\draw[thick] (E) to [out = 270 , in = 45, looseness=.75] (D);
\end{scope}

\begin{scope}
	\begin{pgfinterruptboundingbox} 
		\path [invclip] (B) circle (\gap);
	\end{pgfinterruptboundingbox}
	\draw[thick] (D) to [out = 225 , in = 300] (A) to [out = 120, in = 195] (B);
\end{scope}

\begin{scope}
	\begin{pgfinterruptboundingbox} 
		\path [invclip] (B) circle (\gap);
	\end{pgfinterruptboundingbox}
	\draw[thick] (B) to [out =15, in =165, looseness=.75] (E);
\end{scope}

\begin{scope}
	\begin{pgfinterruptboundingbox} 
		\path [invclip] (D) circle (\gap);
	\end{pgfinterruptboundingbox}
	\draw[thick] (E) to [out =-15, in = 60] (F) to [out = 240,in= -45] (D);
	\end{scope}
%------------------------------------------------------------------	

\draw[thick,->] (5,2.5) -- (12,2.5);
\draw (8.5,2.5) node[above]{\footnotesize{Region crossing change}};
	
\begin{scope}[xshift = 17cm]

\coordinate (A) at (-3.625,.75);
\coordinate (B) at (-2,4);
\coordinate (C) at (0,7);
\coordinate (D) at (0,.5);
\coordinate (E) at (2,4);
\coordinate (F) at (3.625,.75);

\begin{scope}
	\draw[thick] (E) to [out = 90, in = 0] (C) to [out = 180, in = 90] (B);
\end{scope}

\begin{scope}
	\draw[thick] (B) to [out = 270, in = 135, looseness = .75] (D);
\end{scope}

\begin{scope}
	\begin{pgfinterruptboundingbox} 
		\path [invclip] (D) circle (\gap);
	\end{pgfinterruptboundingbox}
	\draw[thick] (E) to [out = 270 , in = 45, looseness=.75] (D);
\end{scope}

\begin{scope}
	\begin{pgfinterruptboundingbox} 
		\path [invclip] (B) circle (\gap);
	\end{pgfinterruptboundingbox}
	\draw[thick] (A) to [out = 120, in = 195] (B);
\end{scope}

\begin{scope}
	\begin{pgfinterruptboundingbox} 
		\path [invclip] (D) circle (\gap);
	\end{pgfinterruptboundingbox}
	\draw[thick] (D) to [out = 225 , in = 300] (A);
\end{scope}

\begin{scope}
	\begin{pgfinterruptboundingbox} 
		\path [invclip] (B) circle (\gap);
	\end{pgfinterruptboundingbox}
	\draw[thick] (B) to [out =15, in =180] (0,4.25);
\end{scope}

\begin{scope}
	\begin{pgfinterruptboundingbox} 
		\path [invclip] (E) circle (\gap);
	\end{pgfinterruptboundingbox}
	\draw[thick] (0,4.25) to [out =0, in =165] (E);
\end{scope}

\begin{scope}
	\begin{pgfinterruptboundingbox} 
		\path [invclip] (E) circle (\gap);
	\end{pgfinterruptboundingbox}
	\draw[thick] (E) to [out =-15, in = 60] (F) to [out = 240,in= -45] (D);
	\end{scope}

\end{scope}

\end{tikzpicture}\]

    \caption{A region crossing change on the shaded region. Since the diagram on the right is of the unknot, the multi-region index of the diagram on the left is two. Inequality \ref{inequality:1} implies that the multi-region index of the trefoil is two.}
    \label{figure:RegionChange}
\end{figure}

Region index can similarly be defined on individual knot diagrams.  An \textit{unknotting region} $R$ of a knot diagram $D$ is a  region of $D$ such that the diagram obtained from a region crossing change at $R$ of $D$ is a diagram of the unknot. The \textit{region index} $\Reg(D)$ of the diagram $D$ is the minimum number of crossings in any unknotting region of $D$ if an unknotting region exists and $\Reg(D)=\infty$ otherwise. Then the region index of a knot $K$ can be defined as $\Reg(K)=\min\{\Reg(D)\;:\; D\text{ is a diagram of }K\}$. Because an unknotting region may or may not minimize the sum in Equation \ref{equation:MRI(D)},  it follows that $\MRI(D)\leq \Reg(D)$ for all knot diagrams $D$. 

The unknotting number $u(D)$ of the knot diagram $D$ is the minimum number of crossing changes needed to change $D$ into a diagram of the unknot, and the \textit{unknotting number} of the knot $K$ is the minimum of $u(D)$ over all diagrams $D$ of $K$. From the definition of multi-region index, it follows that $u(D)\leq \MRI(D)$ for all knot diagrams $D$. Putting the two previous inequalities together and minimizing over all diagrams yields the following inequality that holds for any knot $K$:
\begin{equation}
    \label{inequality:unknotting}
    u(K) \leq \MRI(K) \leq \Reg(K).
\end{equation}
If $R$ is a region that contains just one crossing in its boundary, then a region crossing change of $D$ at $R$ trades one type of Reidemeister 1 twist for the other, as in Figure \ref{r1twist}. Hence for any nontrivial knot $K$,
\begin{equation}
\label{inequality:1}
1<\MRI(K).
\end{equation}
\begin{figure}[h!]
%\centering\includegraphics[scale=.1]{Images/3_1_wR1.jpg}
\[\begin{tikzpicture}[outer sep = 0mm, scale = .4]

\def\gap{.15cm}

\coordinate (A) at (-3.625,.75);
\coordinate (B) at (-2,4);
\coordinate (C) at (0,7);
\coordinate (D) at (0,.5);
\coordinate (E) at (2,4);
\coordinate (F) at (3.75,1.5);
\coordinate (G) at (5.75,1.5);

\fill[white!80!black] (F) to [out = -45, in = 270, looseness=2] (G) to [out = 90, in=45, looseness=2] (F);

\begin{scope}
	\begin{pgfinterruptboundingbox} 
		\path [invclip] (E) circle (\gap);
	\end{pgfinterruptboundingbox}
	\draw[thick] (E) to [out = 90, in = 0] (C) to [out = 180, in = 90] (B);
\end{scope}

\begin{scope}
\begin{pgfinterruptboundingbox} 
		\path [invclip] (D) circle (\gap);
	\end{pgfinterruptboundingbox}
	\draw[thick] (B) to [out = 270, in = 135, looseness = .75] (D);
\end{scope}

\begin{scope}
	\begin{pgfinterruptboundingbox} 
		\path [invclip] (E) circle (\gap);
	\end{pgfinterruptboundingbox}
	\draw[thick] (E) to [out = 270 , in = 45, looseness=.75] (D);
\end{scope}

\begin{scope}
	\begin{pgfinterruptboundingbox} 
		\path [invclip] (B) circle (\gap);
	\end{pgfinterruptboundingbox}
	\draw[thick] (D) to [out = 225 , in = 300] (A) to [out = 120, in = 195] (B);
\end{scope}

\begin{scope}
	\begin{pgfinterruptboundingbox} 
		\path [invclip] (B) circle (\gap);
	\end{pgfinterruptboundingbox}
	\draw[thick] (B) to [out =15, in =165, looseness=.75] (E);
\end{scope}

\begin{scope}
	\begin{pgfinterruptboundingbox} 
		\path [invclip] (F) circle (\gap);
	\end{pgfinterruptboundingbox}
	\draw[thick] (E) to [out =-15, in = 135] (F);
\end{scope}

\begin{scope}
		\begin{pgfinterruptboundingbox} 
		\path [invclip] (F) circle (\gap);
	\end{pgfinterruptboundingbox}
	\draw[thick] (F) to [out = -45, in = 270, looseness=2](G);
\end{scope}

\begin{scope}
		\begin{pgfinterruptboundingbox} 
		\path [invclip] (D) circle (\gap);
	\end{pgfinterruptboundingbox}
	\draw[thick] (G) to [out = 90, in = 45, looseness=2] (F) to [out = 225, in = -45] (D);
\end{scope}

\draw[thick,->] (7,2.5) -- (11,2.5);

%-------------------------------------------------------
\begin{scope}[xshift = 16cm]
\coordinate (A) at (-3.625,.75);
\coordinate (B) at (-2,4);
\coordinate (C) at (0,7);
\coordinate (D) at (0,.5);
\coordinate (E) at (2,4);
\coordinate (F) at (3.75,1.5);
\coordinate (G) at (5.75,1.5);

\begin{scope}
	\begin{pgfinterruptboundingbox} 
		\path [invclip] (E) circle (\gap);
	\end{pgfinterruptboundingbox}
	\draw[thick] (E) to [out = 90, in = 0] (C) to [out = 180, in = 90] (B);
\end{scope}

\begin{scope}
\begin{pgfinterruptboundingbox} 
		\path [invclip] (D) circle (\gap);
	\end{pgfinterruptboundingbox}
	\draw[thick] (B) to [out = 270, in = 135, looseness = .75] (D);
\end{scope}

\begin{scope}
	\begin{pgfinterruptboundingbox} 
		\path [invclip] (E) circle (\gap);
	\end{pgfinterruptboundingbox}
	\draw[thick] (E) to [out = 270 , in = 45, looseness=.75] (D);
\end{scope}

\begin{scope}
	\begin{pgfinterruptboundingbox} 
		\path [invclip] (B) circle (\gap);
	\end{pgfinterruptboundingbox}
	\draw[thick] (D) to [out = 225 , in = 300] (A) to [out = 120, in = 195] (B);
\end{scope}

\begin{scope}
	\begin{pgfinterruptboundingbox} 
		\path [invclip] (B) circle (\gap);
	\end{pgfinterruptboundingbox}
	\draw[thick] (B) to [out =15, in =165, looseness=.75] (E);
\end{scope}

\begin{scope}
	\draw[thick] (E) to [out =-15, in = 135] (F);
\end{scope}

\begin{scope}
	\draw[thick] (F) to [out = -45, in = 270, looseness=2](G);
\end{scope}

\begin{scope}
		\begin{pgfinterruptboundingbox} 
		\path [invclip] (D) circle (\gap);
		\path [invclip] (F) circle (\gap);
	\end{pgfinterruptboundingbox}
	\draw[thick] (G) to [out = 90, in = 45, looseness=2] (F) to [out = 225, in = -45] (D);
\end{scope}
\end{scope}

\end{tikzpicture}\]

\caption{A region crossing change on a Reidemeister 1 twist.}
\label{r1twist}
\end{figure}
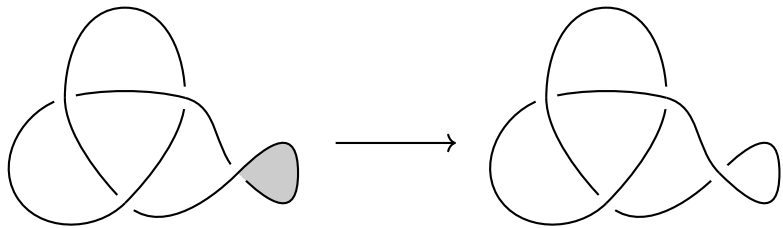

Since each crossing is in the boundary of four regions, a naive upper bound on the multi-region index of a knot $K$ is $4 c(K)$ where $c(K)$ is the crossing number of $K$. Our first main result improves this upper bound on the multi-region index of a knot.
\begin{theorem}
\label{2c}
Let $K$ be a knot with crossing number $c(K)$. Then $\MRI(K)\leq 2c(K)$.
\end{theorem}

Define $mg_2(K)$ to be the minimum number of generators in the first homology of the $2$-fold cyclic branched cover of $S^3$ over $K$. We compute $mg_2(K)$ using the Goeritz matrix of a knot diagram; the quantity $mg_2(K)$ is the number of entries that are greater than one on the diagonal of the Smith normal form of the Goeritz matrix of any diagram of $K$. See Section \ref{section:Goeritz} for more details. Tanaka \cite{Tanaka} proved that $mg_2(K)<\Reg(K)$. Our second main theorem is a version of Tanaka's result for the multi-region index.
\begin{theorem}
\label{mg_2}
For any knot $K$, $mg_2(K)<\MRI(K).$
\end{theorem}
The proof of Theorem \ref{mg_2} follows from an examination of the behavior of Goeritz matrices under crossing changes and region crossing changes. In Section \ref{section:Conclusion}, we use Theorem \ref{mg_2} to show that for any integer $n\geq 2$, there is a knot $K_n$ such that $\MRI(K_n)=n$. Additionally, we compute $\MRI(K)$ for many knots with $9$ or fewer crossings.

This paper is organized as follows. In Section \ref{section:Upper}, we prove Theorem \ref{2c}. In Section \ref{section:Goeritz}, we use the Goeritz matrix of a knot to prove Theorem \ref{mg_2}. Finally, in Section \ref{section:Conclusion}, we compute the multi-region index of some interesting examples and discuss future directions. 

The authors thank Heather Russell for her comments on a draft of this paper.

\section{An upper bound on $\MRI(K)$}
\label{section:Upper}

In this section, we prove Theorem \ref{2c} giving an upper bound for the multi-region index of a knot. A key idea of the proof involves the checkerboard shading of a knot diagram. A checkerboard shading of a knot diagram is an assignment to each region of a color black or white such that no two regions that share an edge in their boundaries are colored the same. Our convention is to assign the unbounded face the color white, and with this convention, the checkerboard shading of a knot diagram is unique. Figure \ref{figure:Checkerboard} shows a checkerboard shading of a diagram of the knot $6_4$. 

A crossing in a knot diagram $D$ is \textit{nugatory} if there exists a simple closed curve in the plane that only intersects $|D|$ at that crossing, and a diagram without nugatory crossings is \textit{reduced}. All minimal crossing diagrams are reduced. Performing region crossing changes on the set of black (or white) regions of a reduced diagram does not change the knot diagram at all. A set of regions where performing a region crossing change of every region in the set does not change the diagram is called an  \textit{ineffective set}. Cliveman, Morris, and Russell \cite{Russell:Ineffective} study ineffective sets for knot and link diagrams.

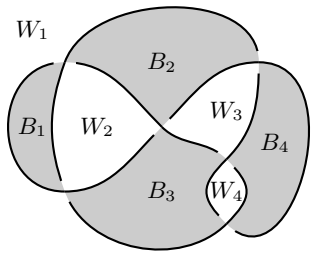
\begin{figure}[h]
\[\begin{tikzpicture}[outer sep = 0mm, scale = .43]

\def\gap{.15cm}

\coordinate (A) at (0,5);
\coordinate (B) at (6,5);
\coordinate (C) at (3,3);
\coordinate (D) at (0,1);
\coordinate (E) at (5,2);
\coordinate (F) at (5,0);

\fill[white!80!black] (D) to [out = 180, in = 180, looseness=1.5] (A) to [out = 250, in =110] (D);

\fill[white!80!black] (A) to [out = 70, in = 90] (B) to [out = 180, in = 45] (C) to [out = 135, in = 0] (A);

\fill[white!80!black]  (F) to [out = -45, in = 0, looseness=1.5] (B)  to [out = 270, in = 45] (E) to [out = -45, in = 45, looseness = 1.5] (F);

\fill[white!80!black]  (F) to [out = 225, in = 290] (D) to [out = 0, in = 225] (C) to [out = -45, in = 135] (E) to [out = 225, in = 135, looseness = 1.5] (F);

\draw (-1,3) node{\footnotesize{$B_1$}};
\draw (3,5) node{\footnotesize{$B_2$}};
\draw (3,1) node{\footnotesize{$B_3$}};
\draw (6.5,2.5) node{\footnotesize{$B_4$}};

\draw (-1,6) node{\footnotesize{$W_1$}};
\draw (1,3) node{\footnotesize{$W_2$}};
\draw (5,3.5) node{\footnotesize{$W_3$}};
\draw (5,1) node{\footnotesize{$W_4$}};

%\fill[black] (A) circle (.1cm);
%\fill[black] (B) circle (.1cm);
%\fill[black] (C) circle (.1cm);
%\fill[black] (D) circle (.1cm);
%\fill[black] (E) circle (.1cm);
%\fill[black] (F) circle (.1cm);

\begin{scope}
	\begin{pgfinterruptboundingbox} 
		\path [invclip] (D) circle (\gap);
	\end{pgfinterruptboundingbox}
	\draw[thick] (F) to [out = 225, in = 290] (D);
	
\end{scope}

\begin{scope}
	\begin{pgfinterruptboundingbox} 
		\path [invclip] (E) circle (\gap);
	\end{pgfinterruptboundingbox}
	\draw[thick] (E) to [out = -45, in = 45, looseness = 1.5] (F);
	
\end{scope}

\begin{scope}
	\begin{pgfinterruptboundingbox} 
		\path [invclip] (E) circle (\gap);
	\end{pgfinterruptboundingbox}
	\draw[thick] (C) to [out = -45, in = 135] (E);
	
\end{scope}

\begin{scope}
	\begin{pgfinterruptboundingbox} 
		\path [invclip] (A) circle (\gap);
	\end{pgfinterruptboundingbox}
	\draw[thick] (A) to [out = 0, in = 135] (C);
	
\end{scope}

\begin{scope}
	\begin{pgfinterruptboundingbox} 
		\path [invclip] (A) circle (\gap);
	\end{pgfinterruptboundingbox}
	\draw[thick] (D) to [out = 180, in = 180, looseness=1.5] (A);
	
\end{scope}

\begin{scope}
	\begin{pgfinterruptboundingbox} 
		\path [invclip] (C) circle (\gap);
	\end{pgfinterruptboundingbox}
	\draw[thick] (C) to [out = 225, in = 0] (D);
	
\end{scope}

\begin{scope}
	\begin{pgfinterruptboundingbox} 
		\path [invclip] (C) circle (\gap);
	\end{pgfinterruptboundingbox}
	\draw[thick] (B) to [out = 180, in = 45] (C);
	
\end{scope}

\begin{scope}
	\begin{pgfinterruptboundingbox} 
		\path [invclip] (F) circle (\gap);
	\end{pgfinterruptboundingbox}
	\draw[thick] (F) to [out = -45, in = 0, looseness=1.5] (B);
	
\end{scope}

\begin{scope}
	\begin{pgfinterruptboundingbox} 
		\path [invclip] (F) circle (\gap);
	\end{pgfinterruptboundingbox}
	\draw[thick] (E) to [out = 225, in = 135, looseness = 1.5] (F);
	
\end{scope}

\begin{scope}
	\begin{pgfinterruptboundingbox} 
		\path [invclip] (B) circle (\gap);
	\end{pgfinterruptboundingbox}
	\draw[thick] (B) to [out = 270, in = 45] (E);
	
\end{scope}

\begin{scope}
	\begin{pgfinterruptboundingbox} 
		\path [invclip] (B) circle (\gap);
	\end{pgfinterruptboundingbox}
	\draw[thick] (A) to [out = 70, in = 90] (B);
	
\end{scope}

\begin{scope}
	\begin{pgfinterruptboundingbox} 
		\path [invclip] (D) circle (\gap);
	\end{pgfinterruptboundingbox}
	\draw[thick] (D) to [out = 110, in = 250] (A);
	
\end{scope}

\end{tikzpicture}\]
\caption{A checkerboard shading of a diagram of the knot $6_4$. The black regions are $\mathcal{B}=\{B_1,B_2, B_3, B_4\}$, and the white regions are $\mathcal{W}=\{W_1,W_2,W_3,W_4\}$.}
\label{figure:Checkerboard}
\end{figure}

\begin{proof}[Proof of Theorem \ref{2c}]
Let $D$ be a minimal crossing diagram of a knot $K$ whose set of checkerboard shaded regions is $\mathcal{R}=\mathcal{B}\sqcup \mathcal{W}$ where $\mathcal{B}$ is the subset of black regions and $\mathcal{W}$ is the set of white regions. Since $D$ is a minimal crossing diagram, it is reduced. Let $B\subseteq \mathcal{B}$ and $W\subseteq \mathcal{W}$ be sets of black and white regions respectively. Define their complements as $B^c=\mathcal{B}-B$ and $W^c=\mathcal{W}-W$. A crossing is in the boundary of two regions in $B$ if and only if it is not in the boundary of any region in $B^c$, and likewise a crossing is in the boundary of exactly one region in $B$ if and only if it is in the boundary of exactly one region in $B^c$. Analogous statements hold for the subset $W$ of white regions and its complement $W^c$. Therefore performing region crossing changes on $D$ at any of the four subsets of regions $B\cup W, B^c\cup W, B\cup W^c$, or $B^c\cup W^c$ results in the same diagram $D'$. 

Now suppose that $B\cup W$ is a set of unknotting regions. Then $B^c\cup W, B\cup W^c$, and $B^c\cup W^c$ are also sets of unknotting regions. Because the total number of crossings in the boundary of all the regions in $\mathcal{B}$ is $2c(D)$, it follows that at least one of $B$ or $B^c$ has at most $c(D)$ crossings in the boundary of all its regions. Similarly, since the total number of crossing in the boundary of all the regions in $\mathcal{W}$ is $2c(D)$, it follows that at least one of $W$ or $W^c$ has at most $c(D)$ crossings in the boundary of all its regions. Therefore at least one of the four sets of unknotting regions $B\cup W, B^c\cup W, B\cup W^c$, or $B^c\cup W^c$ has a total of at most $2c(D)$ crossings in the boundary of its regions.

\end{proof}
Figure \ref{figure:2c_1} shows an example of the algorithm described in the proof of Theorem \ref{2c} for the standard diagram of the knot $6_4$. Performing region crossing changes on the black regions $B_1$, $B_2$, and $B_3$ produces the same diagram as performing a region crossing change on the black region $B_4$. The sum in Equation \ref{equation:MRI(D)} for the set of unknotting regions on the left of Figure \ref{figure:2c_1} is 12 while the sum for the set of unknotting regions on the right of Figure \ref{figure:2c_1} is 6.

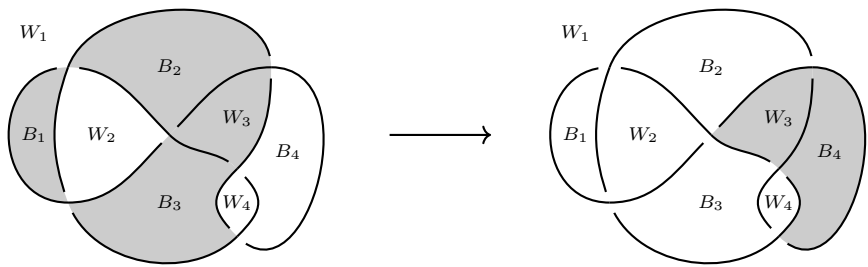
\begin{figure}[h!]
\[\begin{tikzpicture}[outer sep = 0mm, scale = .45]

\def\gap{.15cm}

\coordinate (A) at (0,5);
\coordinate (B) at (6,5);
\coordinate (C) at (3,3);
\coordinate (D) at (0,1);
\coordinate (E) at (5,2);
\coordinate (F) at (5,0);

\fill[white!80!black] (D) to [out = 180, in = 180, looseness=1.5] (A) to [out = 250, in =110] (D);

\fill[white!80!black] (A) to [out = 70, in = 90] (B) to [out = 180, in = 45] (C) to [out = 135, in = 0] (A);

\fill[white!80!black] (B) to [out = 180, in = 45] (C) to [out = -45, in = 135] (E) to [out = 45, in = 270] (B);

%\fill[white!80!black]  (F) to [out = -45, in = 0, looseness=1.5] (B)  to [out = 270, in = 45] (E) to [out = -45, in = 45, looseness = 1.5] (F);

\fill[white!80!black]  (F) to [out = 225, in = 290] (D) to [out = 0, in = 225] (C) to [out = -45, in = 135] (E) to [out = 225, in = 135, looseness = 1.5] (F);

\draw (-1,3) node{\tiny{$B_1$}};
\draw (3,5) node{\tiny{$B_2$}};
\draw (3,1) node{\tiny{$B_3$}};
\draw (6.5,2.5) node{\tiny{$B_4$}};

\draw (-1,6) node{\tiny{$W_1$}};
\draw (1,3) node{\tiny{$W_2$}};
\draw (5,3.5) node{\tiny{$W_3$}};
\draw (5,1) node{\tiny{$W_4$}};

%\fill[black] (A) circle (.1cm);
%\fill[black] (B) circle (.1cm);
%\fill[black] (C) circle (.1cm);
%\fill[black] (D) circle (.1cm);
%\fill[black] (E) circle (.1cm);
%\fill[black] (F) circle (.1cm);

\begin{scope}
	\begin{pgfinterruptboundingbox} 
		\path [invclip] (D) circle (\gap);
	\end{pgfinterruptboundingbox}
	\draw[thick] (F) to [out = 225, in = 290] (D);
	
\end{scope}

\begin{scope}
	\begin{pgfinterruptboundingbox} 
		\path [invclip] (E) circle (\gap);
	\end{pgfinterruptboundingbox}
	\draw[thick] (E) to [out = -45, in = 45, looseness = 1.5] (F);
	
\end{scope}

\begin{scope}
	\begin{pgfinterruptboundingbox} 
		\path [invclip] (E) circle (\gap);
	\end{pgfinterruptboundingbox}
	\draw[thick] (C) to [out = -45, in = 135] (E);
	
\end{scope}

\begin{scope}
	\begin{pgfinterruptboundingbox} 
		\path [invclip] (A) circle (\gap);
	\end{pgfinterruptboundingbox}
	\draw[thick] (A) to [out = 0, in = 135] (C);
	
\end{scope}

\begin{scope}
	\begin{pgfinterruptboundingbox} 
		\path [invclip] (A) circle (\gap);
	\end{pgfinterruptboundingbox}
	\draw[thick] (D) to [out = 180, in = 180, looseness=1.5] (A);
	
\end{scope}

\begin{scope}
	\begin{pgfinterruptboundingbox} 
		\path [invclip] (C) circle (\gap);
	\end{pgfinterruptboundingbox}
	\draw[thick] (C) to [out = 225, in = 0] (D);
	
\end{scope}

\begin{scope}
	\begin{pgfinterruptboundingbox} 
		\path [invclip] (C) circle (\gap);
	\end{pgfinterruptboundingbox}
	\draw[thick] (B) to [out = 180, in = 45] (C);
	
\end{scope}

\begin{scope}
	\begin{pgfinterruptboundingbox} 
		\path [invclip] (F) circle (\gap);
	\end{pgfinterruptboundingbox}
	\draw[thick] (F) to [out = -45, in = 0, looseness=1.5] (B);
	
\end{scope}

\begin{scope}
	\begin{pgfinterruptboundingbox} 
		\path [invclip] (F) circle (\gap);
	\end{pgfinterruptboundingbox}
	\draw[thick] (E) to [out = 225, in = 135, looseness = 1.5] (F);
	
\end{scope}

\begin{scope}
	\begin{pgfinterruptboundingbox} 
		\path [invclip] (B) circle (\gap);
	\end{pgfinterruptboundingbox}
	\draw[thick] (B) to [out = 270, in = 45] (E);
	
\end{scope}

\begin{scope}
	\begin{pgfinterruptboundingbox} 
		\path [invclip] (B) circle (\gap);
	\end{pgfinterruptboundingbox}
	\draw[thick] (A) to [out = 70, in = 90] (B);
	
\end{scope}

\begin{scope}
	\begin{pgfinterruptboundingbox} 
		\path [invclip] (D) circle (\gap);
	\end{pgfinterruptboundingbox}
	\draw[thick] (D) to [out = 110, in = 250] (A);
	
\end{scope}

\draw[thick, ->] (9.5,3) -- (12.5,3);

%---------------------
\begin{scope}[xshift = 16 cm]
\coordinate (A) at (0,5);
\coordinate (B) at (6,5);
\coordinate (C) at (3,3);
\coordinate (D) at (0,1);
\coordinate (E) at (5,2);
\coordinate (F) at (5,0);

%\fill[white!80!black] (D) to [out = 180, in = 180, looseness=1.5] (A) to [out = 250, in =110] (D);

%\fill[white!80!black] (A) to [out = 70, in = 90] (B) to [out = 180, in = 45] (C) to [out = 135, in = 0] (A);

\fill[white!80!black] (B) to [out = 180, in = 45] (C) to [out = -45, in = 135] (E) to [out = 45, in = 270] (B);

\fill[white!80!black]  (F) to [out = -45, in = 0, looseness=1.5] (B)  to [out = 270, in = 45] (E) to [out = -45, in = 45, looseness = 1.5] (F);

%\fill[white!80!black]  (F) to [out = 225, in = 290] (D) to [out = 0, in = 225] (C) to [out = -45, in = 135] (E) to [out = 225, in = 135, looseness = 1.5] (F);

\draw (-1,3) node{\tiny{$B_1$}};
\draw (3,5) node{\tiny{$B_2$}};
\draw (3,1) node{\tiny{$B_3$}};
\draw (6.5,2.5) node{\tiny{$B_4$}};

\draw (-1,6) node{\tiny{$W_1$}};
\draw (1,3) node{\tiny{$W_2$}};
\draw (5,3.5) node{\tiny{$W_3$}};
\draw (5,1) node{\tiny{$W_4$}};

%\fill[black] (A) circle (.1cm);
%\fill[black] (B) circle (.1cm);
%\fill[black] (C) circle (.1cm);
%\fill[black] (D) circle (.1cm);
%\fill[black] (E) circle (.1cm);
%\fill[black] (F) circle (.1cm);

\begin{scope}
	\begin{pgfinterruptboundingbox} 
		\path [invclip] (D) circle (\gap);
	\end{pgfinterruptboundingbox}
	\draw[thick] (F) to [out = 225, in = 290] (D);
	
\end{scope}

\begin{scope}
	\begin{pgfinterruptboundingbox} 
		\path [invclip] (E) circle (\gap);
	\end{pgfinterruptboundingbox}
	\draw[thick] (E) to [out = -45, in = 45, looseness = 1.5] (F);
	
\end{scope}

\begin{scope}
	\begin{pgfinterruptboundingbox} 
		\path [invclip] (E) circle (\gap);
	\end{pgfinterruptboundingbox}
	\draw[thick] (C) to [out = -45, in = 135] (E);
	
\end{scope}

\begin{scope}
	\begin{pgfinterruptboundingbox} 
		\path [invclip] (A) circle (\gap);
	\end{pgfinterruptboundingbox}
	\draw[thick] (A) to [out = 0, in = 135] (C);
	
\end{scope}

\begin{scope}
	\begin{pgfinterruptboundingbox} 
		\path [invclip] (A) circle (\gap);
	\end{pgfinterruptboundingbox}
	\draw[thick] (D) to [out = 180, in = 180, looseness=1.5] (A);
	
\end{scope}

\begin{scope}
	\begin{pgfinterruptboundingbox} 
		\path [invclip] (C) circle (\gap);
	\end{pgfinterruptboundingbox}
	\draw[thick] (C) to [out = 225, in = 0] (D);
	
\end{scope}

\begin{scope}
	\begin{pgfinterruptboundingbox} 
		\path [invclip] (C) circle (\gap);
	\end{pgfinterruptboundingbox}
	\draw[thick] (B) to [out = 180, in = 45] (C);
	
\end{scope}

\begin{scope}
	\begin{pgfinterruptboundingbox} 
		\path [invclip] (F) circle (\gap);
	\end{pgfinterruptboundingbox}
	\draw[thick] (F) to [out = -45, in = 0, looseness=1.5] (B);
	
\end{scope}

\begin{scope}
	\begin{pgfinterruptboundingbox} 
		\path [invclip] (F) circle (\gap);
	\end{pgfinterruptboundingbox}
	\draw[thick] (E) to [out = 225, in = 135, looseness = 1.5] (F);
	
\end{scope}

\begin{scope}
	\begin{pgfinterruptboundingbox} 
		\path [invclip] (B) circle (\gap);
	\end{pgfinterruptboundingbox}
	\draw[thick] (B) to [out = 270, in = 45] (E);
	
\end{scope}

\begin{scope}
	\begin{pgfinterruptboundingbox} 
		\path [invclip] (B) circle (\gap);
	\end{pgfinterruptboundingbox}
	\draw[thick] (A) to [out = 70, in = 90] (B);
	
\end{scope}

\begin{scope}
	\begin{pgfinterruptboundingbox} 
		\path [invclip] (D) circle (\gap);
	\end{pgfinterruptboundingbox}
	\draw[thick] (D) to [out = 110, in = 250] (A);
	
\end{scope}

\end{scope}

\end{tikzpicture}\]
\caption{Replacing the shaded regions $B_1$, $B_2$, and $B_3$ with $B_4$ yields the same diagram after performing region crossing changes and decreases the sum in Equation \ref{equation:MRI(D)}.}
\label{figure:2c_1}
\end{figure}

Among knot diagrams in Rolfsen's table of knots with ten or fewer crossings, every diagram except one has $\MRI(D)<c(D)$. The standard diagram $D_{9_{35}}$ of the knot $9_{35}$ has $\MRI(D_{9_{35}}) = 9$, although we do not know if the standard diagram minimizes the multi-region index over all diagrams of $9_{35}$. See Figure \ref{figure:9_35}. 
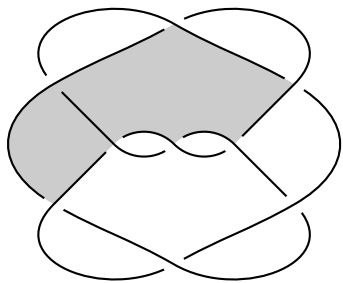
\begin{figure}[h]
\[\begin{tikzpicture}[scale = .4]

\def\gap{.15cm}

\coordinate (A) at (0,2);
\coordinate (B) at (0,6);
\coordinate (C) at (2,4);
\coordinate (D) at (4,0);
\coordinate (E) at (4,4);
\coordinate (F) at (4,8);
\coordinate (G) at (6,4);
\coordinate (H) at (8,2);
\coordinate (I) at (8,6);

\fill[white!80!black] (A) to [out = 45, in=225] (C) to [out = 45, in = 135] (E) to [out = 45, in = 135] (G) to [out =45, in = 225] (I) to [out = 150, in = -30] (F) to [out = 210, in = 30] (B) to [out = 210, in = 150, looseness = 1.5] (A);

\begin{scope}
	\begin{pgfinterruptboundingbox} 
		\path [invclip] (H) circle (\gap);
		\path [invclip] (E) circle (\gap);
	\end{pgfinterruptboundingbox}
	\draw[thick] (H) to [out = 135, in = -45, looseness = 1] (G) to [out = 135, in = 45, looseness =1] (E);	
\end{scope}

\begin{scope}
	\begin{pgfinterruptboundingbox} 
		\path [invclip] (H) circle (\gap);
		\path [invclip] (A) circle (\gap);
	\end{pgfinterruptboundingbox}
	\draw[thick] (A) to [out = -30, in = 150, looseness = 1] (D) to [out = -30, in = -45, looseness =1.5] (H);	
\end{scope}

\begin{scope}
	\begin{pgfinterruptboundingbox} 
		\path [invclip] (F) circle (\gap);
		\path [invclip] (A) circle (\gap);
	\end{pgfinterruptboundingbox}
	\draw[thick] (F) to [out = 210, in = 30, looseness = 1] (B) to [out = 210, in = 150, looseness =1.5] (A);	
\end{scope}

\begin{scope}
	\begin{pgfinterruptboundingbox} 
		\path [invclip] (F) circle (\gap);
		\path [invclip] (G) circle (\gap);
	\end{pgfinterruptboundingbox}
	\draw[thick] (G) to [out = 45, in = 225, looseness = 1] (I) to [out = 45, in = 30, looseness =1.5] (F);	
\end{scope}

\begin{scope}
	\begin{pgfinterruptboundingbox} 
		\path [invclip] (C) circle (\gap);
		\path [invclip] (G) circle (\gap);
	\end{pgfinterruptboundingbox}
	\draw[thick] (C) to [out = 45, in = 135, looseness = 1] (E) to [out = -45, in = 225, looseness =1] (G);	
\end{scope}

\begin{scope}
	\begin{pgfinterruptboundingbox} 
		\path [invclip] (C) circle (\gap);
		\path [invclip] (D) circle (\gap);
	\end{pgfinterruptboundingbox}
	\draw[thick] (D) to [out = 210, in = 225, looseness = 1.5] (A) to [out = 45, in = 225, looseness =1] (C);	
\end{scope}

\begin{scope}
	\begin{pgfinterruptboundingbox} 
		\path [invclip] (I) circle (\gap);
		\path [invclip] (D) circle (\gap);
	\end{pgfinterruptboundingbox}
	\draw[thick] (I) to [out = -30, in = 30, looseness = 1.5] (H) to [out = 210, in = 30, looseness =1] (D);	
\end{scope}

\begin{scope}
	\begin{pgfinterruptboundingbox} 
		\path [invclip] (I) circle (\gap);
		\path [invclip] (B) circle (\gap);
	\end{pgfinterruptboundingbox}
	\draw[thick] (B) to [out = 135, in = 150, looseness = 1.5] (F) to [out = -30, in = 150, looseness =1] (I);	
\end{scope}

\begin{scope}
	\begin{pgfinterruptboundingbox} 
		\path [invclip] (E) circle (\gap);
		\path [invclip] (B) circle (\gap);
	\end{pgfinterruptboundingbox}
	\draw[thick] (E) to [out = 225, in = -45, looseness = 1] (C) to [out = 135, in = -45, looseness =1] (B);	
\end{scope}

%\fill[black] (A) circle (.1cm);
%\fill[black] (B) circle (.1cm);
%\fill[black] (C) circle (.1cm);
%\fill[black] (D) circle (.1cm);
%\fill[black] (E) circle (.1cm);
%\fill[black] (F) circle (.1cm);
%\fill[black] (G) circle (.1cm);
%\fill[black] (H) circle (.1cm);
%\fill[black] (I) circle (.1cm);

\end{tikzpicture}\]

\caption{The standard diagram $D_{9_{35}}$ of the knot $9_{35}$ has $\MRI(D_{9_{35}})=9$. The shaded regions are a set of unknotting regions that minimize the sum in Equation \ref{equation:MRI(D)}.}
\label{figure:9_35}
\end{figure}

It is possible that a better upper bound than $2c(K)$ exists for the multi-region index of a knot. As the proof of Theorem \ref{2c} shows, if $D$ and $D'$ are two diagrams with the same projection (i.e. $|D|=|D'|$), then it is possible to transform $D$ into $D'$ via a sequence of region crossing changes where the total number of crossings in the boundary of the regions being changed is at most $2c(D)$. So the fact that $D'$ is a diagram of the unknot is not used in a crucial way in our proof. Perhaps finding a better upper bound for the multi-region index of a knot will make use of this fact.

\section{The Goeritz matrix and $mg_2(K)$}
\label{section:Goeritz}
In this section, we recall the construction of the Goeritz matrix of a knot diagram, show how to use the Goeritz matrix to compute $mg_2(K)$, and prove Theorem \ref{mg_2}.

Let $K$ be a knot with diagram $D$. Checkerboard shade the regions in $D$ so that regions sharing a crossing, but not a side, are shaded. Let $R_0, \dots, R_m$ be the shaded regions. The index of a crossing $\zeta(c)$ is the value of $\pm 1$, as shown in Figure \ref{figure:CrossingType}. 

%figure
\begin{figure}[h!]
    \centering
   \[\begin{tikzpicture}[thick]

\begin{scope}[xshift = 5cm]
\fill[white!80!black] (2,2) -- (1,1) -- (0,2);
\fill[white!80!black] (0,0) -- (1,1) -- (2,0);
\draw (0,0) -- (2,2);
\draw (0,2) -- (.7,1.3);
\draw (1.3,.7) -- (2,0);
\draw (1,-.5) node{$\zeta(c) = -1$};
\end{scope}

\fill[white!80!black] (2,2) -- (1,1) -- (0,2);
\fill[white!80!black] (0,0) -- (1,1) -- (2,0);
\draw (0,2) -- (2,0);
\draw (0,0) -- (.7,.7);
\draw (1.3,1.3) -- (2,2);
\draw (1,-.5) node{$\zeta(c) = 1$};

\end{tikzpicture}\]
    \caption{The index $\zeta(c)$ at a crossing c}
    \label{figure:CrossingType}
\end{figure}
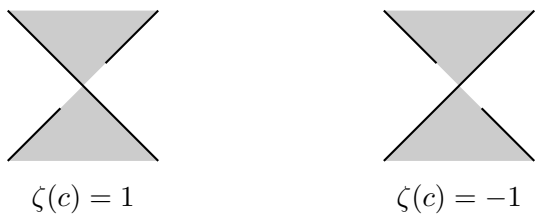

Define the pre-Goeritz matrix $\overline{G}_D$ to be the symmetric $(m+1) \times (m+1)$ matrix whose ($i, j$)th entry is $\sum \zeta(c)$ where the sum is over all crossings in the boundary of regions $R_i$ and $R_j$ and whose $(i,i)$th diagonal entry is the negative of the sum of all the other entries in the $i$th column. The \textit{Goeritz matrix} of $D$, denoted $G_D$, is the $m \times m$ matrix obtained by deleting the first column and first row of $\overline{G}_D$. 

%3_1#3_1 example
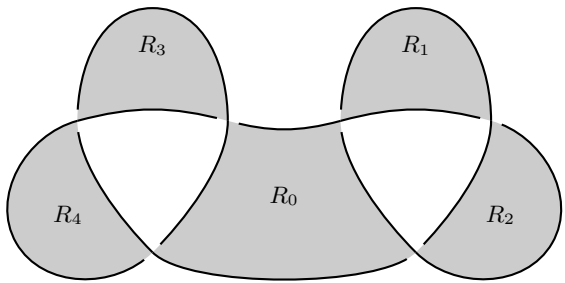
\begin{figure}[h!]
    \centering
    \[\begin{tikzpicture}[outer sep = 0mm, scale = .5]

\def\gap{.15cm}

\coordinate (A) at (-3.625,.75);
\coordinate (B) at (-2,4);
\coordinate (C) at (0,7);
\coordinate (D) at (0,.5);
\coordinate (E) at (2,4);
\coordinate (F) at (3.625,.75);
\coordinate (G) at (5,4);
\coordinate (H) at (9,4);
\coordinate (I) at (7,7);
\coordinate (J) at (7,.5);
\coordinate (K) at (10.625,.75);

\fill[white!80!black] (E) to [out = 90, in = 0] (C) to [out = 180, in = 90] (B)  to [out = 15, in = 165] (E);
\fill[white!80!black] (B) to [out = 270, in = 135, looseness = .75] (D) to [out = 225 , in = 300] (A) to [out = 120, in = 195] (B);
\fill[white!80!black] (J) to [out = -45, in = 240] (K) to [out = 60, in = -15] (H)  to [out = 270, in =45, looseness=.75] (J);
\fill[white!80!black] (H) to [out = 90 , in = 0] (I) to [out = 180, in = 90] (G) to [out = 15, in = 165] (H);

\fill[white!80!black] (D) to [out = -45 , in = 225, looseness=.5] (J) to [out = 135, in = 270, looseness = .75] (G) to [out = 195, in = -15] (E) to  [out =270 , in = 45, looseness=.75] (D);

\draw (3.5,2) node{\footnotesize{$R_0$}};
\draw (7,6) node{\footnotesize{$R_1$}};
\draw (9.25,1.5) node{\footnotesize{$R_2$}};
\draw (0,6) node{\footnotesize{$R_3$}};
\draw (-2.25,1.5) node{\footnotesize{$R_4$}};

%\fill[black] (A) circle (.1cm);
%\fill[black] (B) circle (.1cm);
%\fill[black] (C) circle (.1cm);
%\fill[black] (D) circle (.1cm);
%\fill[black] (E) circle (.1cm);
%\fill[black] (G) circle (.1cm);
%\fill[black] (H) circle (.1cm);
%\fill[black] (I) circle (.1cm);
%\fill[black] (J) circle (.1cm);
%\fill[black] (K) circle (.1cm);

%\fill[white!80!black] (E) to [out = -15, in = 60] (F) to [out = 240, in=-45] (D) to [out = 45, in = 270, looseness = .75] (E);

\begin{scope}
	\begin{pgfinterruptboundingbox} 
		\path [invclip] (B) circle (\gap);
	\end{pgfinterruptboundingbox}
	\draw[thick] (E) to [out = 90, in = 0] (C) to [out = 180, in = 90] (B);
\end{scope}

\begin{scope}
\begin{pgfinterruptboundingbox} 
		\path [invclip] (B) circle (\gap);
	\end{pgfinterruptboundingbox}
	\draw[thick] (B) to [out = 270, in = 135, looseness = .75] (D);
\end{scope}

\begin{scope}
	\begin{pgfinterruptboundingbox} 
		\path [invclip] (J) circle (\gap);
	\end{pgfinterruptboundingbox}
	\draw[thick] (D) to [out = -45 , in = 225, looseness=.5] (J);
\end{scope}

\begin{scope}
	\begin{pgfinterruptboundingbox} 
		\path [invclip] (J) circle (\gap);
	\end{pgfinterruptboundingbox}
	\draw[thick] (J) to [out = 45 , in = 270, looseness=.75] (H);
\end{scope}

\begin{scope}
	\begin{pgfinterruptboundingbox} 
		\path [invclip] (G) circle (\gap);
	\end{pgfinterruptboundingbox}
	\draw[thick] (H) to [out = 90 , in = 0] (I) to [out = 180, in = 90] (G);
\end{scope}

\begin{scope}
	\begin{pgfinterruptboundingbox} 
		\path [invclip] (G) circle (\gap);
	\end{pgfinterruptboundingbox}
	\draw[thick] (G) to [out = 270, in = 135, looseness = .75] (J);
\end{scope}

\begin{scope}
	\begin{pgfinterruptboundingbox} 
		\path [invclip] (H) circle (\gap);
	\end{pgfinterruptboundingbox}
	\draw[thick] (J) to [out = -45, in = 240] (K) to [out = 60, in = -15] (H);
\end{scope}

\begin{scope}
	\begin{pgfinterruptboundingbox} 
		\path [invclip] (H) circle (\gap);
	\end{pgfinterruptboundingbox}
	\draw[thick] (H) to [out = 165, in = 15] (G);
\end{scope}

\begin{scope}
	\begin{pgfinterruptboundingbox} 
		\path [invclip] (E) circle (\gap);
	\end{pgfinterruptboundingbox}
	\draw[thick] (G) to [out = 195, in = -15] (E);
\end{scope}

\begin{scope}
	\begin{pgfinterruptboundingbox} 
		\path [invclip] (E) circle (\gap);
	\end{pgfinterruptboundingbox}
	\draw[thick] (E) to [out = 165, in = 15] (B);
\end{scope}
%-----
\begin{scope}
	\begin{pgfinterruptboundingbox} 
		\path [invclip] (D) circle (\gap);
	\end{pgfinterruptboundingbox}
	\draw[thick] (D) to [out = 225 , in = 300] (A) to [out = 120, in = 195] (B);
\end{scope}

\begin{scope}
	\begin{pgfinterruptboundingbox} 
		\path [invclip] (D) circle (\gap);
	\end{pgfinterruptboundingbox}
	\draw[thick] (D) to [out =45 , in = 270, looseness=.75] (E);
\end{scope}

%\begin{scope}
%	\begin{pgfinterruptboundingbox} 
%		\path [invclip] (B) circle (\gap);
%	\end{pgfinterruptboundingbox}
%	\draw[thick] (B) to [out =15, in =165, looseness=.75] (E);
%\end{scope}

%\begin{scope}
%	\begin{pgfinterruptboundingbox} 
%		\path [invclip] (D) circle (\gap);
%	\end{pgfinterruptboundingbox}
%	\draw[thick] (E) to [out =-15, in = 60] (F) to [out = 240,in= -45] (D);
%	\end{scope}
\end{tikzpicture}\]

    \caption{A checkerboard shading of a diagram of the connected sum of two trefoils.}
    \label{connectsum}
\end{figure}

Figure \ref{connectsum} shows a checkerboard shading of a diagram $D$ of the connected sum of two trefoils. Its pre-Goeritz matrix $\overline{G}_D$ and Goeritz matrix $G_D$ are
\[\overline{G}_D =
\begin{blockarray}{cccccc}
 & R_0 & R_1 & R_2 & R_3 & R_4 \\
\begin{block}{c[rrrrr]}
  R_0 & 4 & -1 & -1 & -1 & -1 \\
  R_1 & -1 & 2 & -1 & 0 & 0\\ 
  R_2 & -1 & -1 & 2 & 0 & 0 \\
  R_3 & -1 & 0 & 0 & 2 & -1 \\
  R_4 & -1 & 0 & 0 & -1 & 2\\
\end{block}
\end{blockarray}
\quad \text{and} \quad G_D = 
\begin{blockarray}{ccccc}
 & R_1 & R_2 & R_3 & R_4 \\
\begin{block}{c[rrrr]}
  R_1  & 2 & -1 & 0 & 0\\ 
  R_2  & -1 & 2 & 0 & 0 \\
  R_3  & 0 & 0 & 2 & -1 \\
  R_4  & 0 & 0 & -1 & 2\\
\end{block}
\end{blockarray}.
\]

The invariant $mg_2(K)$ is computed from the Smith normal form of the Goeritz matrix $G_D$. There are invertible $m \times m$ matrices $S$ and $T$ with integer entries such that 
\[S G_D T=\begin{bmatrix}
\alpha_1 & 0 & 0 & & \cdots & & 0 \\
0 & \alpha_2 & 0 & & \cdots & & 0 \\
0 & 0 & \ddots & & & & 0\\
\vdots & & & \alpha_r & & & \vdots \\
 & & & & 0 & & \\
 & & & & & \ddots &  \\
0 & & & \cdots & & & 0
\end{bmatrix},\]
where each $\alpha_k$ is a positive integer satisfying $\alpha_k | \alpha_{k+1}$ for $1\leq k < r$. The matrix $SG_DT$ is called the \textit{Smith normal form} of $G_D$.

A matrix can be transformed into its Smith normal form via a sequence of the following row and column operations:
\begin{enumerate}[label=\roman*.]
    \item replacing row or column $i$ with $i + tj$, where $j$ is another row or column respectively and $t$ is an integer,
    \item switching rows or columns, and
    \item scaling rows or columns by $\pm 1$.
\end{enumerate}

The Goeritz matrix is a presentation matrix for the first homology of branched double cover of $S^3$ along $K$, and therefore $mg_2(K)$ is the number of entries on the diagonal of the Smith normal form of $G_D$ that are greater than one. For example, the Smith normal form of the Goeritz matrix $G_D$ for the diagram in Figure \ref{connectsum} is
\[\left[ \begin{array}{c c c c}
1 & 0 & 0 & 0\\
0 & 1 & 0 & 0\\
0 & 0 & 3 & 0\\
0 & 0 & 0 & 3
\end{array}\right],\]
and thus $mg_2(K) = 2$ for the knot in Figure \ref{connectsum}.

Wendt \cite{Wendt} proved that $mg_2(K)\leq u(K)$ for any knot $K$. We prove the following related result using Goeritz matrices.
\begin{theorem}
\label{abs val dk}
If $J$ and $K$ are knots with diagrams $D_J$ and $D_K$ that differ by a single crossing change, then $|mg_2(J) - mg_2(K)|\leq 1$.
\end{theorem}

\begin{proof}
Suppose that $c$ is the crossing in $D_J$ that when changed produces the diagram $D_K$. A local isotopy near the crossing $c$ can transform the diagram as in Figure \ref{fig:tanakapic}, and so we assume that $D_J$ is a diagram that looks like the diagram on the right of Figure \ref{fig:tanakapic} near the crossing $c$. The diagram $D_J$ can be checkerboard shaded so that $\zeta(c)=-1$ in $D_J$ and $\zeta(c)=1$ in $D_K$.

Assume $mg_2(K)=m$ and $mg_2(J)=n$. The Goeritz matrices of $D_K$ and $D_J$ have the following forms:
%matrix
\[G_{D_K} =
  \begin{bmatrix}
    \begin{array}{c|ccccc}
  0 & 0 & -1 & 0 & \cdots & 0\\ \hline
  0 & & & & &\\
  -1 & & & & &\\
  0 & & & \text{\Large{$X$}}& &\\
\vdots & & & & & \\
  0 & & & & &\\
    \end{array}
  \end{bmatrix}~\text{and}~
  G_{D_J} = 
  \begin{bmatrix}
    \begin{array}{c|ccccc}
  2 & 0 & -1 & 0 & \cdots & 0\\ \hline
  0 & & & & &\\
  -1 & & & & &\\
  0 & & & \text{\Large{$X$}}& &\\
  \vdots & & & & & \\
  0 & & & & &\\
    \end{array}
  \end{bmatrix}.
\]
\\
A sequence of row and column operations transforms $G_{D_K}$ into
%matrix
\[G'_{D_K} =
\begin{bmatrix}
\begin{array}{c | c c c}
1  & 0 & \cdots & 0 \\ \hline
0  & & & \\
\vdots  & & \text{\Large{$X_{1}$}}& \\
0  & & & \\
\end{array}
\end{bmatrix}.
\]
The same sequence of row and column operations transforms $G_{D_J}$ into

%matrix
\[G'_{D_J} =
\begin{bmatrix} 
\begin{array}{c | c c c}
x' & & \text{\Large{*}}  & \\ \hline
&  & & \\
\text{\Large{*}}  & & \text{\Large{$X_{1}$}}& \\
& & & \\
\end{array}
\end{bmatrix}.
\] 
Since $mg_2(K)=n$, we can transform $G'_{D_J}$ into

%matrix
\[G''_{D_J} = 
\begin{bmatrix} 
\begin{array}{c | c c c | c c c}
x & &  \text{\Large{*}} & & 0 &\cdots & 0 \\ \hline
& \alpha_1 & &  &  & &\\
\text{\Large{*}} & & \ddots & & & 0\\
& & & \alpha_n & & & \\ \hline
0 & & & & 1 & & \\
\vdots & & 0  & & & \ddots &\\
0 & & & & & & 1\\
\end{array}
\end{bmatrix}
\]
where $\alpha_i>1$ for $1\leq i\leq n$. Therefore, $m=mg_2(J)\leq n+1$. A similar argument shows that $mg_2(K)\leq m+1$. Since $m\leq n+1$ and $n\leq m+1$, it follows that $|n-m|\leq 1$.

\end{proof}

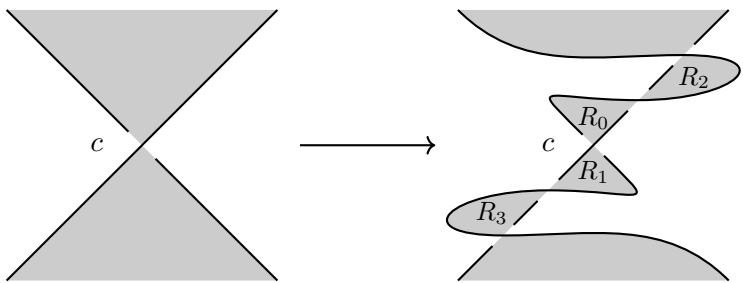
\begin{figure}[h!]
    \centering
    \[\begin{tikzpicture}[thick, scale = .6]

\fill[white!80!black] (0,0) -- (3,3) -- (6,0);
\fill[white!80!black] (0,6) -- (3,3) -- (6,6);

\draw (2,3) node{$c$};
\draw (0,0) -- (6,6);
\draw (0,6) -- (2.7,3.3);
\draw (3.3,2.7) -- (6,0);

\draw[thick, ->] (6.5,3) -- (9.5,3);

\begin{scope}[xshift = 10cm]

\coordinate (A) at (0,0);
\coordinate (B) at (1,1);
\coordinate (C) at (2,2);
\coordinate (D) at (3,3);
\coordinate (E) at (4,4);
\coordinate (F) at (5,5);
\coordinate (G) at (6,6);

\draw (2,3) node{$c$};

\fill[white!80!black] (0,6) to [out = -45, in= 180] (F) -- (G);
\fill[white!80!black] (F) to [out = 0, in = 0, looseness = 4] (E);
\fill[white!80!black] (E) to [out = 180, in = 135, looseness = 4](D);
\fill[white!80!black] (D) to [out = -45, in = 0, looseness = 4] (C);
\fill[white!80!black] (C) to [out = 180, in= 180, looseness = 4] (B);
\fill[white!80!black] (A)  -- (B) to [out = 0, in = 135] (6,0);

\draw ( 3,3.6) node{\small{$R_0$}};
\draw (3,2.4) node{\small{$R_1$}};
\draw (5.23,4.5) node{\small{$R_2$}};
\draw (.75,1.5) node{\small{$R_3$}};

\begin{scope}
	\begin{pgfinterruptboundingbox} 
		\path [invclip] (D) circle (.2cm);
	\end{pgfinterruptboundingbox}
	\draw[thick] (0,6) to [out =-45 , in = 180, looseness=1]
	(5,5) to [out = 0, in = 0, looseness=4]
	(4,4) to [out = 180, in = 135, looseness = 4]
	(3,3) to [out = -45, in = 0, looseness = 4]
	(2,2) to [out = 180, in = 180, looseness = 4]
	(1,1) to [out = 0, in = 135, looseness = 1]
	(6,0);
	
\end{scope}

\begin{scope}
	\begin{pgfinterruptboundingbox} 
		\path [invclip] (B) circle (.2cm);
		\path [invclip] (C) circle (.2cm);
		\path [invclip] (E) circle (.2cm);
		\path [invclip] (F) circle (.2cm);
	\end{pgfinterruptboundingbox}
	\draw[thick] (0,0) -- (6,6);
	
\end{scope}

\end{scope}

\end{tikzpicture}\]
    \caption{An isotopy near the crossing $c$ makes the Goeritz matrices of $D_J$ and $D_K$ have the desired format.}
    \label{fig:tanakapic}
\end{figure}

Tanaka \cite{Tanaka} proved the following theorem about $mg_2(K)$ and the region index of a knot.
\begin{theorem}[Tanaka]
\label{theorem:Tanaka}
For any knot $K$, $mg_2(K)<\Reg(K).$
\end{theorem}

Theorems \ref{abs val dk} and \ref{theorem:Tanaka} are now used to prove Theorem \ref{mg_2}.
\begin{proof}[Proof of Theorem \ref{mg_2}]
Let $D$ be a diagram of $K$ such that $\MRI(D)=\MRI(K)$. Then $D$ contains a set of unknotting regions $\{R_{i_1},\dots,R_{i_k}\}$ such that 
\[\MRI(K)=\sum_{j=1}^k c(R_{i_j})\]
where $c(R_{i_j})$ is the number of crossings in the boundary of the region $R_{i_j}$.  

Let $D_1$ be the diagram obtained from $D$ by performing region crossing changes on $R_{i_1}, \dots, R_{i_{k-1}}$, and let $K_1$ be the knot with diagram $D_1$. Then $R_{i_k}$ is an unknotting region of $D_1$ with $c(R_{i_k})$ crossings in its boundary, and thus $\Reg(K_1) \leq c(R_{i_k})$. Theorem \ref{theorem:Tanaka} implies that $mg_2(K_1)<\Reg(K_1)\leq c(R_{i_k})$

The diagram $D$ can be obtained from the diagram $D_1$ via $\sum_{j=1}^{k-1} c(R_{i_j})$ crossing changes. Theorem \ref{abs val dk} implies that 
\[\mid mg_2(K) - mg_2(K_1)\mid  \leq \sum_{j=1}^{k-1} c(R_{i_j}).\]
Thus \[mg_2(K) < \sum_{j=1}^{k} c(R_{i_j}) = \MRI(K).\]
\end{proof}

\section{Computations and open questions}
\label{section:Conclusion}

In this section, we compute the multi-region index for some interesting families of knots. We also discuss some natural open questions related to the multi-region index.

In the following example, we show there is an infinite family of knots $K_i$ each with a pair of minimal crossing diagrams $D_i$ and $D_i'$ such that the multi-region index of $D_i$ is different than the multi-region index of $D_i'$ (and likewise for the region indices of $D_i$ and $D_i'$). The diagrams $D_i$ and $D_i'$ can be transformed into one another via a flype. The unknotting number of a diagram is preserved by flyping, but the following example shows the same is not true for the region index and multi-region index of a diagram. 

\begin{example} For each positive integer $i$, define $K_i$ to be the knot with diagrams $D_i$ and $D_i'$, as depicted in Figure   \ref{figure:InfiniteFlype}. There are $2i+1$ crossings in the twist regions of $D_i$ and $D_i'$, and the knot $K_0$ is $8_{15}$. In each diagram, the shaded region is an unknotting region. One can show that these shaded regions minimize the sum in Equation \ref{equation:MRI(D)} by performing region crossing changes on sets of regions with fewer crossings, computing the determinant of the resulting knot, and seeing that those determinants are not one. Therefore, $\MRI(D_i)=\Reg(D_i)=3$ while $\MRI(D_i')=\Reg(D_i')=4$. 
\end{example}
%8_15 2k+1 image
\begin{figure}[h!]
\[\begin{tikzpicture}[outer sep = 0mm, scale = .3]

\def\gap{.15cm}

\coordinate (A) at (0,8);
\coordinate (B) at (5,8);
\coordinate (C) at (8,10);
\coordinate (D) at (8,6);
\coordinate (E) at (15,7);
\coordinate (F) at (12,5);
\coordinate (G) at (10, 2);
\coordinate (H) at (8,0);
\coordinate (I) at (12,0);

\fill[white!80!black] (D) to [out = 272, in = 133] (G) to [out = 33, in= 270] (F) to [out = 186, in = -23] (D);

\draw [
    thick,
    decoration={
        brace,
        mirror,
        amplitude = 10pt,
        raise=0.5cm
    },
    decorate
] (0,7.6) -- (5,7.6);
\draw (2.5,4.5) node{\footnotesize{$2i+1$}};

\fill (2.5,8) circle (.1cm);
\fill (2,8) circle (.1cm);
\fill (3,8) circle (.1cm);

\begin{scope}
		\begin{pgfinterruptboundingbox} 
		\path [invclip] (A) circle (\gap);
	\end{pgfinterruptboundingbox}
	\draw[thick] (A) to [out = 45, in = 210, looseness=1] (1,9);
\end{scope}

\begin{scope}
		\begin{pgfinterruptboundingbox} 
		\path [invclip] (C) circle (\gap);
	\end{pgfinterruptboundingbox}
	\draw[thick] (1,7) to [out angle = 150, in angle =120, curve through={(A)}, ]  (C);
\end{scope}

\begin{scope}
	\begin{pgfinterruptboundingbox} 
		\path [invclip] (B) circle (\gap);
		\path [invclip] (E) circle (\gap);
	\end{pgfinterruptboundingbox}
	\draw [thick] (4,7) to [out angle = 30, in angle = 60, curve through = {(B) .. (C)}] (E);
\end{scope}

\begin{scope}
	\begin{pgfinterruptboundingbox} 
		\path [invclip] (D) circle (\gap);
		\path [invclip] (E) circle (\gap);
	\end{pgfinterruptboundingbox}
	\draw [thick] (4,9) to [out angle = -30,  in angle = 240, curve through={(B) .. (D) .. (F)}] (E);
\end{scope}

\begin{scope}
	\begin{pgfinterruptboundingbox} 
		\path [invclip] (F) circle (\gap);
		\path [invclip] (I) circle (\gap);
	\end{pgfinterruptboundingbox}
	\draw [thick] (F) to [out angle = 90, in angle =15, curve through = {(E)}] (I);
\end{scope}	

\begin{scope}
	\begin{pgfinterruptboundingbox} 
		\path [invclip] (A) circle (\gap);
		\path [invclip] (I) circle (\gap);
	\end{pgfinterruptboundingbox}
	\draw[thick] (1,9) to [out angle = 210, in angle = 195, curve through = {(A) .. (H)}] (I);

\end{scope}

\begin{scope}
	\begin{pgfinterruptboundingbox} 
		\path [invclip] (F) circle (\gap);
		\path [invclip] (H) circle (\gap);
	\end{pgfinterruptboundingbox}
	\draw[thick] (F) to [out angle = 270, in angle = 90, curve through = {(G)}] (H);
\end{scope}

\begin{scope}
	\begin{pgfinterruptboundingbox} 
		\path [invclip] (G) circle (\gap);
		\path [invclip] (H) circle (\gap);
	\end{pgfinterruptboundingbox}
	\draw[thick] (G) to [out angle = 315, in angle = 270, curve through = {(I)}] (H);

\end{scope}

\begin{scope}
	\begin{pgfinterruptboundingbox} 
		\path [invclip] (C) circle (\gap);
		\path [invclip] (G) circle (\gap);
	\end{pgfinterruptboundingbox}
	\draw[thick] (C) to [out angle = 300, in angle = 135, curve through = {(D)}](G);

\end{scope}

\draw(7.5,-2.5) node{$D_i$};

\begin{scope}[xshift = 22cm]
\draw(7.5, -2.5) node{$D_i'$};
\coordinate (A) at (0,8);
\coordinate (B) at (5,8);
\coordinate (C) at (8,10);
\coordinate (D) at (8,6);
\coordinate (E) at (15,7);
\coordinate (F) at (12,5);
\coordinate (G) at (9,3);
\coordinate (H) at (11,3);
\coordinate (I) at (10,0);

%\fill[black] (A) circle (.1cm);
%\fill[black] (B) circle (.1cm);
%\fill[black] (C) circle (.1cm);
%\fill[black] (D) circle (.1cm);
%\fill[black] (E) circle (.1cm);
%\fill[black] (F) circle (.1cm);
%\fill[black] (G) circle (.1cm);
%\fill[black] (H) circle (.1cm);
%\fill[black] (I) circle (.1cm);

\draw [
    thick,
    decoration={
        brace,
        mirror,
        amplitude = 10pt,
        raise=0.5cm
    },
    decorate
] (0,7.6) -- (5,7.6);
\draw (2.5,4.5) node{\footnotesize{$2i+1$}};

\fill (2.5,8) circle (.1cm);
\fill (2,8) circle (.1cm);
\fill (3,8) circle (.1cm);

\fill[white!80!black]  (E)  to [out = 260, in = -2]  (F) to [out = 255, in = 35] (H) to [out = -40, in = 35] (I) to [out = -45, in = -15, looseness=2.2]  (E);

\begin{scope}
	\path[clip] (1,9) to [out angle = 210, in angle = -45, curve through = {(A) .. (I)}] (11.5,3) to [out = 25, in = 270] (12.25,4.75) to [out=-2, in =260]  (E) -- (20,7)  -- (20,-2) -- (1,-2) -- (1,9);
	\fill[white!80!black](I) to [out angle = -45, in angle = 80, curve through = {(E)}] (F);
\end{scope}

\begin{scope}
	\begin{pgfinterruptboundingbox} 
		\path [invclip] (C) circle (\gap);
	\end{pgfinterruptboundingbox}
	\draw[thick] (1,7) to [out angle =150, in angle = 120, curve through = {(A)}](C);

\end{scope}

\begin{scope}
	\begin{pgfinterruptboundingbox} 
		\path [invclip] (A) circle (\gap);
		\path [invclip] (H) circle (\gap);
	\end{pgfinterruptboundingbox}
	\draw[thick] (1,9) to [out angle = 210, in angle = -45, curve through = {(A) .. (I)}] (H);

\end{scope}

\begin{scope}
	\begin{pgfinterruptboundingbox} 
		\path [invclip] (I) circle (\gap);
		\path [invclip] (H) circle (\gap);
	\end{pgfinterruptboundingbox}
	\draw[thick] (H) to [out angle = 45, in angle = 135, curve through = {(G)}] (I);

\end{scope}

\begin{scope}
	\begin{pgfinterruptboundingbox} 
		\path [invclip] (I) circle (\gap);
		\path [invclip] (F) circle (\gap);
	\end{pgfinterruptboundingbox}
	\draw[thick] (I) to [out angle = -45, in angle = 80, curve through = {(E)}] (F);

\end{scope}

\begin{scope}
	\begin{pgfinterruptboundingbox} 
		\path [invclip] (G) circle (\gap);
		\path [invclip] (F) circle (\gap);
	\end{pgfinterruptboundingbox}
	\draw[thick] (F) to [out angle = 260, in angle = -45, curve through = {(H)}] (G);
\end{scope}

\begin{scope}
	\begin{pgfinterruptboundingbox} 
		\path [invclip] (G) circle (\gap);
		\path [invclip] (C) circle (\gap);
	\end{pgfinterruptboundingbox}
	\draw[thick] (G) to [out angle = 135, in angle = -60, curve through = {(D)}] (C);
\end{scope}

\begin{scope}
	\begin{pgfinterruptboundingbox} 
		\path [invclip] (B) circle (\gap);
		\path [invclip] (E) circle (\gap);
	\end{pgfinterruptboundingbox}
	\draw[thick] (4,7) to [out angle = 30, in angle = 80, curve through = {(B) .. (C)}] (E);
\end{scope}

\begin{scope}
	\begin{pgfinterruptboundingbox} 
		\path [invclip] (D) circle (\gap);
	\end{pgfinterruptboundingbox}
	\draw[thick] (4,9) to [out angle = -30, in angle = 170, curve through = {(B)}] (D);
	
\end{scope}

\begin{scope}
	\begin{pgfinterruptboundingbox} 
		\path [invclip] (D) circle (\gap);
		\path [invclip] (E) circle (\gap);
	\end{pgfinterruptboundingbox}
	\draw[thick] (D) to [out angle = -10, in angle = 260, curve through = {(F)}] (E);
	
\end{scope}

\end{scope}
\end{tikzpicture}\]

\caption{The region index and the multi-region index of a diagram can change under flypes.}
\label{figure:InfiniteFlype}
\end{figure}
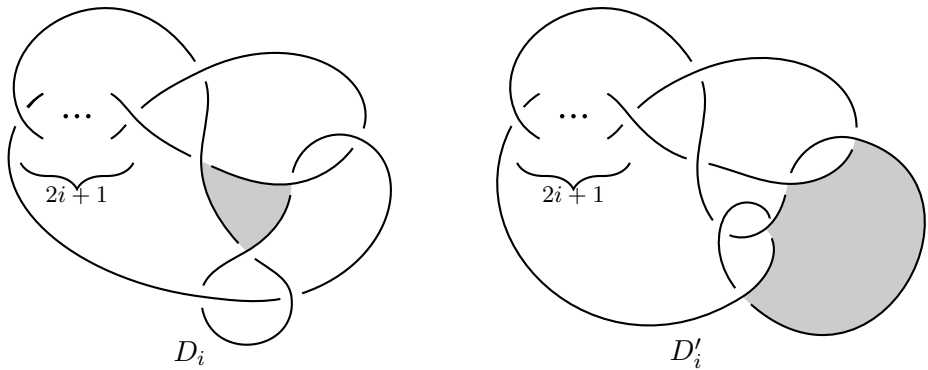

Tanaka \cite{Tanaka} found an infinite family of knots $K_n$ where $n$ is a positive integer such that the region index of $K_n$ is $n+1$. We remark that the same technique yields $\MRI(K_n)=n+1$.
\begin{example}
\label{example:Tanaka}
Let $T_{2,3}$ denote the $(2,3)$-torus knot, i.e. the right-handed trefoil, and let $T_{2,-3}$ denote the $(2,-3)$-torus knot, i.e. the left-handed trefoil. If $n=2k$ is even, define $K_n=kT_{2,3}\# kT_{2,-3}$, that is $K_n$ is the connected sum of $k$ copies of $T_{2,3}$ and $k$ copies of $T_{2,-3}$. If $n=2k+1$ is odd, then let $K_n= (k+1)T_{2,3}\# k T_{2,-3}$. Tanaka showed that $mg_2(K_n)=n$ and that $\Reg(K_n)=n+1$. Therefore $\MRI(K_n)=n+1$.
\end{example}

In the following example, we use Inequality \ref{inequality:unknotting} to compute the multi-region index of many two-stranded torus knots.
\begin{example}
\label{example:torus}
The unknotting number of the $(2,2k+1)$-torus knot $T_{2,2k+1}$ is $u(T_{2,2k+1})=k$. The standard diagram of $T_{2,2k+1}$ has two regions with all $2k+1$ crossings in their boundary and $2k+1$ regions with two crossings each in their boundary. Performing a region crossing change on $\lceil k/2 \rceil$ regions with two crossings each in their boundary where none of the regions changed share a crossing in their boundaries results in a diagram of the unknot. Therefore, if $k\geq 2$, then 
\[ \MRI(T_{2,2k+1}) = \begin{cases}k & \text{if $k$ is even,}\\
k~\text{or}~k+1&\text{if $k$ is odd.}
\end{cases}
\]
When $k=1$, Inequality \ref{inequality:1} implies that $\MRI(T_{2,3})=2$. See Figures \ref{figure:table1} and \ref{figure:table2} for the specific examples of $T_{2,3}$, $T_{2,5}$, and $T_{2,9}$ (knots $3_1, 5_1,$ and $9_1$ respectively). 
\end{example}

The following example shows that there is a diagram where the region index is finite and the multi-region index is strictly less than the region index.
\begin{example}
\label{example:8_1}
Let $D_{8_1}$ be the diagram of the knot $8_1$ depicted in Figure \ref{figure:8_1}. The shaded regions of the diagram on the left of Figure \ref{figure:8_1} minimize Equation \ref{equation:MRI(D)} and show that $\MRI(D_{8_1})=5$. The shaded region of diagram on the right of Figure \ref{figure:8_1} is an unknotting region with as few crossings as possible in its boundary. Thus $\Reg(D_{8_1})=7$.
\end{example}
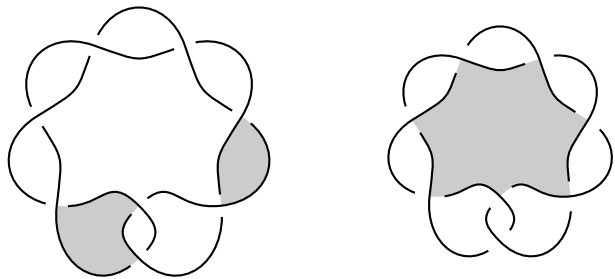
\begin{figure}[h]
\[\begin{tikzpicture}[thick]
\begin{scope}[outer sep = 0mm, scale = .7]

\def\gap{.15cm}

\coordinate (A) at (270:2.25);
\coordinate (B) at (270:1.25);
\coordinate (C) at (270+360/7:2);
\coordinate (D) at (270 + 2*360/7:2);
\coordinate (E) at (270 + 3*360/7:2);
\coordinate (F) at (270 + 4*360/7:2);
\coordinate (G) at (270 + 5*360/7:2);
\coordinate (H) at (270 + 6*360/7:2);

%\fill[black] (A) circle (.1cm);
%\fill[black] (B) circle (.1cm);
%\fill[black] (C) circle (.1cm);
%\fill[black] (D) circle (.1cm);
%\fill[black] (E) circle (.1cm);
%\fill[black] (F) circle (.1cm);
%\fill[black] (G) circle (.1cm);
%\fill[black] (H) circle (.1cm);

\fill[white!80!black, looseness=1.5] (B) to [out = 225, in = 135] (A) to [out =225, in =  315+6*360/7] (H) to [out = 45+6*360/7, in = 135] (B);

\fill[white!80!black, looseness=1.5] (C) to [out = 315+360/7, in = 225+2*360/7] (D) to [out =135+2*360/7, in = 45+360/7] (C);

\begin{scope}[looseness=1.5]
	\begin{pgfinterruptboundingbox} 
		\path [invclip] (A) circle (\gap);
		\path [invclip] (H) circle (\gap);
	\end{pgfinterruptboundingbox}
	\draw[thick] (H) to [out = 45+6*360/7, in = 135] (B) to [out = -45, in = 45] (A);
\end{scope}

\begin{scope}[looseness=1.5]
	\begin{pgfinterruptboundingbox} 
		\path [invclip] (A) circle (\gap);
		\path [invclip] (G) circle (\gap);
	\end{pgfinterruptboundingbox}
	\draw[thick] (G) to [out = 45+5*360/7, in = 135+6*360/7] (H) to [out = 315+6*360/7, in = 225] (A);
\end{scope}

\begin{scope}[looseness=1.5]
	\begin{pgfinterruptboundingbox} 
		\path [invclip] (F) circle (\gap);
		\path [invclip] (H) circle (\gap);
	\end{pgfinterruptboundingbox}
	\draw[thick] (F) to [out = 45+4*360/7, in = 135+5*360/7] (G) to [out = 315+5*360/7, in = 225+6*360/7] (H);
\end{scope}

\begin{scope}[looseness=1.5]
	\begin{pgfinterruptboundingbox} 
		\path [invclip] (E) circle (\gap);
		\path [invclip] (G) circle (\gap);
	\end{pgfinterruptboundingbox}
	\draw[thick] (E) to [out = 45+3*360/7, in = 135+4*360/7] (F) to [out = 315+4*360/7, in = 225+5*360/7] (G);
\end{scope}

\begin{scope}[looseness=1.5]
	\begin{pgfinterruptboundingbox} 
		\path [invclip] (F) circle (\gap);
		\path [invclip] (D) circle (\gap);
	\end{pgfinterruptboundingbox}
	\draw[thick] (D) to [out = 45+2*360/7, in = 135+3*360/7] (E) to [out = 315+3*360/7, in = 225+4*360/7] (F);
\end{scope}

\begin{scope}[looseness=1.5]
	\begin{pgfinterruptboundingbox} 
		\path [invclip] (E) circle (\gap);
		\path [invclip] (C) circle (\gap);
	\end{pgfinterruptboundingbox}
	\draw[thick] (C) to [out = 45+360/7, in = 135+2*360/7] (D) to [out = 315+2*360/7, in = 225+3*360/7] (E);
\end{scope}

\begin{scope}[looseness=1.5]
	\begin{pgfinterruptboundingbox} 
		\path [invclip] (B) circle (\gap);
		\path [invclip] (D) circle (\gap);
	\end{pgfinterruptboundingbox}
	\draw[thick] (B) to [out = 45, in = 135+360/7] (C) to [out = 315+360/7, in = 225+2*360/7] (D);
\end{scope}

\begin{scope}[looseness=1.5]
	\begin{pgfinterruptboundingbox} 
		\path [invclip] (B) circle (\gap);
		\path [invclip] (C) circle (\gap);
	\end{pgfinterruptboundingbox}
	\draw[thick] (B) to [out = 225, in = 135] (A) to [out = -45, in = 225+360/7] (C);
\end{scope}

\end{scope}

%%%%%%%%%%%%%%%%%%%%%%%%
\begin{scope}[outer sep = 0mm, scale = .6, xshift = 8cm]

\def\gap{.2cm}

\coordinate (A) at (270:2.25);
\coordinate (B) at (270:1.25);
\coordinate (C) at (270+360/7:2);
\coordinate (D) at (270 + 2*360/7:2);
\coordinate (E) at (270 + 3*360/7:2);
\coordinate (F) at (270 + 4*360/7:2);
\coordinate (G) at (270 + 5*360/7:2);
\coordinate (H) at (270 + 6*360/7:2);

%\fill[black] (A) circle (.1cm);
%\fill[black] (B) circle (.1cm);
%\fill[black] (C) circle (.1cm);
%\fill[black] (D) circle (.1cm);
%\fill[black] (E) circle (.1cm);
%\fill[black] (F) circle (.1cm);
%\fill[black] (G) circle (.1cm);
%\fill[black] (H) circle (.1cm);

\fill[white!80!black, looseness=1.5] (B) to [out = 45, in = 135+360/7] (C) to [out = 45+360/7, in = 135+2*360/7] (D) to [out = 45+2*360/7, in = 135+3*360/7] (E) to [out = 45+3*360/7, in = 135+4*360/7] (F) to [out = 45+4*360/7, in = 135+5*360/7] (G) to [out = 45+5*360/7, in = 135+6*360/7] (H)  to [out = 45+6*360/7, in = 135] (B);

\begin{scope}[looseness=1.5]
	\begin{pgfinterruptboundingbox} 
		\path [invclip] (A) circle (\gap);
		\path [invclip] (H) circle (\gap);
	\end{pgfinterruptboundingbox}
	\draw[thick] (H) to [out = 45+6*360/7, in = 135] (B) to [out = -45, in = 45] (A);
\end{scope}

\begin{scope}[looseness=1.5]
	\begin{pgfinterruptboundingbox} 
		\path [invclip] (A) circle (\gap);
		\path [invclip] (G) circle (\gap);
	\end{pgfinterruptboundingbox}
	\draw[thick] (G) to [out = 45+5*360/7, in = 135+6*360/7] (H) to [out = 315+6*360/7, in = 225] (A);
\end{scope}

\begin{scope}[looseness=1.5]
	\begin{pgfinterruptboundingbox} 
		\path [invclip] (F) circle (\gap);
		\path [invclip] (H) circle (\gap);
	\end{pgfinterruptboundingbox}
	\draw[thick] (F) to [out = 45+4*360/7, in = 135+5*360/7] (G) to [out = 315+5*360/7, in = 225+6*360/7] (H);
\end{scope}

\begin{scope}[looseness=1.5]
	\begin{pgfinterruptboundingbox} 
		\path [invclip] (E) circle (\gap);
		\path [invclip] (G) circle (\gap);
	\end{pgfinterruptboundingbox}
	\draw[thick] (E) to [out = 45+3*360/7, in = 135+4*360/7] (F) to [out = 315+4*360/7, in = 225+5*360/7] (G);
\end{scope}

\begin{scope}[looseness=1.5]
	\begin{pgfinterruptboundingbox} 
		\path [invclip] (F) circle (\gap);
		\path [invclip] (D) circle (\gap);
	\end{pgfinterruptboundingbox}
	\draw[thick] (D) to [out = 45+2*360/7, in = 135+3*360/7] (E) to [out = 315+3*360/7, in = 225+4*360/7] (F);
\end{scope}

\begin{scope}[looseness=1.5]
	\begin{pgfinterruptboundingbox} 
		\path [invclip] (E) circle (\gap);
		\path [invclip] (C) circle (\gap);
	\end{pgfinterruptboundingbox}
	\draw[thick] (C) to [out = 45+360/7, in = 135+2*360/7] (D) to [out = 315+2*360/7, in = 225+3*360/7] (E);
\end{scope}

\begin{scope}[looseness=1.5]
	\begin{pgfinterruptboundingbox} 
		\path [invclip] (B) circle (\gap);
		\path [invclip] (D) circle (\gap);
	\end{pgfinterruptboundingbox}
	\draw[thick] (B) to [out = 45, in = 135+360/7] (C) to [out = 315+360/7, in = 225+2*360/7] (D);
\end{scope}

\begin{scope}[looseness=1.5]
	\begin{pgfinterruptboundingbox} 
		\path [invclip] (B) circle (\gap);
		\path [invclip] (C) circle (\gap);
	\end{pgfinterruptboundingbox}
	\draw[thick] (B) to [out = 225, in = 135] (A) to [out = -45, in = 225+360/7] (C);
\end{scope}

\end{scope}

\end{tikzpicture}\]
\caption{The shaded regions of the diagram $D_{8_1}$ of $8_1$ on the left minimize Equation \ref{equation:MRI(D)}. The shaded region of the diagram $D_{8_1}$ of $8_1$ on the right is an unknotting region with as few crossings as possible for this diagram.}
\label{figure:8_1}
\end{figure}

In Example \ref{example:Tanaka}, we have $\MRI(K_n)=\Reg(K_n)=n+1$. On the other hand, in Examples \ref{example:torus} and \ref{example:8_1}, we have diagrams where the multi-region index is strictly less than the region index. The knots in Examples \ref{example:torus} and \ref{example:8_1} are candidates for the following open question.
\begin{question}
Is there a knot with $\MRI(K) < \Reg(K)?$ Can the difference $\Reg(K) - \MRI(K)$ be arbitrarily large?
\end{question}
In light of this question, we have the following result.

\begin{theorem}
\label{mri<reg}
If $K$ is a knot such that $\MRI(K) < \Reg(K)$, then $\MRI(K)\geq 4$ and $\Reg(K) \geq 5$.
\end{theorem}
\begin{proof}
Performing a region crossing change on a region with only one crossing in its boundary does not change the knot type. Thus if $\MRI(K)=2$, then $K$ has a diagram where changing one region with two crossings in its boundary transforms $K$ into the unknot (and not two regions each with one crossing in their boundaries). Similarly, if $\MRI(K)=3$, then $K$ has a diagram where changing one region with three crossings in its boundary transforms $K$ into the unknot (and not three regions each with one crossing in their boundaries or two regions, one with two crossings and one with one crossing in their boundaries). Therefore $\MRI(K)=2$ if and only if $\Reg(K)=2$, and $\MRI(K)=3$ if and only if $\Reg(K)=3$, implying the result.
\end{proof}

Because $mg_2(K)\leq u(K)\leq \MRI(K)$ and $mg_2(K)<\MRI(K)$ are our main tools for computing the multi-region index of a knot, it is difficult to find knots whose unknotting number and multi-region index are far apart.
\begin{question}
Can the differences $\MRI(K)-u(K)$ and $\Reg(K) - u(K)$ be arbitrarily far apart?
\end{question}

Kawauchi, Kishimoto, and Shimizu \cite{KKS} proved the following theorem about the region index of a knot.
\begin{theorem}[Kawauchi, Kishimoto, Shimizu]
\label{KKSTheorem} 
Let $\sigma(K)$ be the signature of the knot $K$, and let $a_2(K)$ be the coefficient of the quadratic term of the Conway polynomial of $K$. If $\sigma(K)=0$ and $a_2(K)\equiv 1\mod 2$, then $\Reg(K)\geq 3$.
\end{theorem}

Because $\MRI(K)=2$ if and only if $\Reg(K)=2$, it follows that if $\sigma(K)=0$ and $a_2(K)\equiv 1 \mod 2$, then $\MRI(K) \geq 3$.

\begin{example}
Theorems \ref{mg_2} and \ref{KKSTheorem} and Inequalities \ref{inequality:unknotting} and \ref{inequality:1} allow us to compute the multi-region index of 36 of the 84 nontrivial prime knots with nine or fewer crossings. In Figures \ref{figure:table1}, \ref{figure:table2}, and \ref{figure:table3}, we shade the regions where region crossing changes should be performed to obtain the unknot. If the multi-region index of a diagram is two, then it is automatically minimal. If the multi-region index of a diagram is greater than two and we can conclude it is minimal via Inequality \ref{inequality:unknotting}, that is if $u(K)=\MRI(K)$, then we mark that computation with a dagger $\dagger$. If the multi-region index of a diagram is greater than two and we can conclude it is minimal via Theorem \ref{mg_2}, that is if $mg_2(K) + 1 = \MRI(K)$, then we mark that computation with an asterisk $\ast$. If the multi-region index of a diagram is three and we can conclude it is minimal via Theorem \ref{KKSTheorem}, that is if $\sigma(K)=0$ and $a_2(K)\equiv 1 \mod 2$, then we mark that computation with a double dagger $\ddagger$. Computations of unknotting numbers, signatures, and Conway polynomials were obtained from KnotInfo \cite{KnotInfo}.
\end{example}

\begin{figure}
\[
\]
\caption{The multi-region index of small knots. If $u(K)=\MRI(K)$, then the entry is marked with $\dagger$. If $mg_2(K) + 1 = \MRI(K)$, then the entry is marked with $\ast$. If $\sigma(K)=0$, $a_2(K)\equiv 1 \mod 2$, and $\MRI(K)=3$, then the entry is marked with $\ddagger$.}
\label{figure:table1}
\end{figure}

\begin{figure}
\[%
\]
\caption{The multi-region index of small knots. If $u(K)=\MRI(K)$, then the entry is marked with $\dagger$. If $mg_2(K) + 1 = \MRI(K)$, then the entry is marked with $\ast$. If $\sigma(K)=0$, $a_2(K)\equiv 1 \mod 2$, and $\MRI(K)=3$, then the entry is marked with $\ddagger$.}
\label{figure:table2}
\end{figure}

\begin{figure}
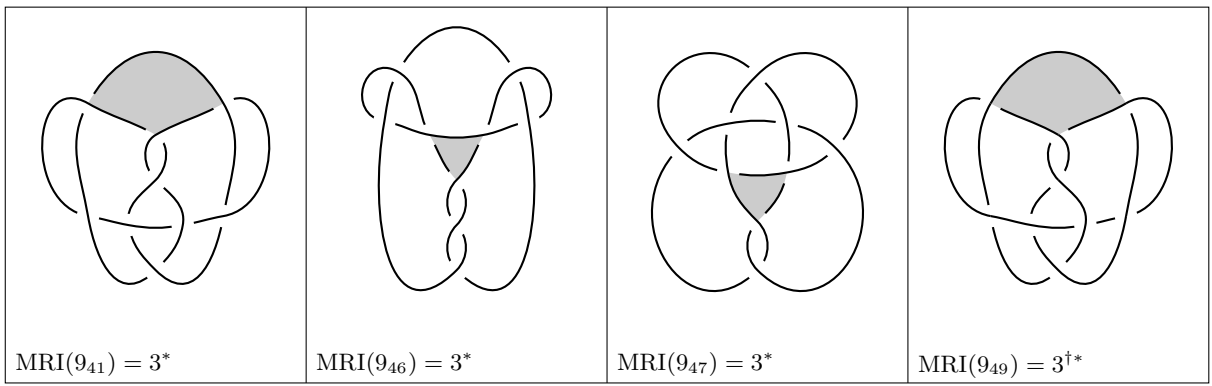

\[%
\]
\caption{The multi-region index of small knots. If $u(K)=\MRI(K)$, then the entry is marked with $\dagger$. If $mg_2(K) + 1 = \MRI(K)$, then the entry is marked with $\ast$. If $\sigma(K)=0$, $a_2(K)\equiv 1 \mod 2$, and $\MRI(K)=3$, then the entry is marked with $\ddagger$.}
\label{figure:table3}
\end{figure}

In order to compute the region index and the multi-region index of more knots, it would be useful to have more lower bounds. One potential lower bound comes from the $n$-fold cyclic branched cover of $S^3$ along $K$. Let $mg_n(K)$ be the minimum number of generators of the first homology of the $n$-fold cyclic branched cover of $S^3$ along $K$. Wendt \cite{Wendt} proved that the unknotting number $u(K)$ has lower bound given by $\frac{mg_n(K)}{n-1}\leq u(K)$. \begin{question}
What are other lower bounds for the region index and multi-region index of a knot? In particular, does the inequality
\[\frac{mg_n(K)}{n-1}<\MRI(K)\]
hold for all knots $K$?
\end{question}

\newpage

\bibliography{mribib}{}
\bibliographystyle {amsalpha}

\end{document}